\documentclass[12pt,reqno]{amsart} 
\usepackage{amssymb}
\usepackage[reqno]{amsmath}
\usepackage{amsthm}
\usepackage{b}
\usepackage[a4paper]{geometry}
\usepackage{esint}
\usepackage{empheq}
\usepackage{color}
\usepackage{units}
\usepackage{xfrac}
\usepackage{thmtools}

\declaretheoremstyle[bodyfont=\normalfont]{normalbody}

\theoremstyle{break}

\newtheorem{thrm}{Theorem}
\newtheorem{coro}{Corollary}

\newtheorem{assumption}{Assumption}

\declaretheorem[style=normalbody]{Algorithm}

\begin{document}

\title[MLMC McKean-Vlasov]{A Multilevel Monte Carlo method for a class of McKean-Vlasov processes}
\author{L.F. Ricketson}
\address{Courant Institute, New York University, New York, NY 11226}
\email{ricketson@cims.nyu.edu}
\date{\today}

\maketitle

\begin{abstract}
We generalize the multilevel Monte Carlo (MLMC) method of Giles to the simulation of systems of particles that interact via a mean field.  When the number of particles is large, these systems are described by a McKean-Vlasov process - a stochastic differential equation (SDE) whose coefficients depend on expectations of the solution as well as pathwise data.  In contrast to standard MLMC, the new method uses mean field estimates at coarse levels to inform the fine level computations.  Using techniques from the theory of propagation of chaos, we prove convergence and complexity results for the algorithm in a special case.  We find that the new method achieves $L^1$ error of size $\varepsilon$ with $O(\varepsilon^{-2} (\log \varepsilon)^5)$ complexity, in contrast to the $O(\varepsilon^{-3})$ complexity of standard methods.  We also prove a variance scaling result that strongly suggests similar performance of the algorithm in a more general context.  We present numerical examples from applications and observe the expected behavior in each case.  
\end{abstract}

\section{Introduction}
The multilevel Monte Carlo (MLMC) method introduced in \cite{giles2008multilevel} has proven to be a powerful tool for simulation of stochastic differential equations (SDEs) and related models.  The method and its variants have found applications in finance \cite{ben2014multilevel,belomestny2013pricing,burgos2012computing,giles2009multilevel} (to name a few), biochemical kinetics \cite{anderson2012multilevel,anderson2014complexity}, plasma physics \cite{ricketson2014two,rosin2014multilevel}, and porous media flow \cite{efendiev2015multilevel,ketelsen2013hierarchical,muller2013multilevel,muller2014solver}, among others.  

In short, the method is applied to a stochastic differential equation (SDE)
\begin{equation} \label{SDE}
	dX_t = a(X_t,t) dt + b(X_t,t) dW_t
\end{equation}
by introducing a hierarchy of time step ``levels" $\Delta t_\ell \propto 2^{-\ell}$ and taking advantage of correlations between simulations at adjacent levels to achieve variance reduction.  Typically, the goal is to compute the expectation of some functional of the SDE's solution.  Standard MLMC achieves this with root-mean-square error $\varepsilon$ in $O(\varepsilon^{-2}(\log \varepsilon)^2)$ time, a dramatic improvement over the $O(\varepsilon^{-3})$ time required by naive Monte Carlo.  

However, the modeling power of SDEs - and all the models used in the references above - is limited by the absence of interaction between distinct realizations.  In a wide array of applications, a large number of particles or agents are not only subject to some external stochastic forcing, but also interactions amongst themselves.  Examples of this behavior abound in physics, including kinetic descriptions of plasmas and rarefied gases, as well as ferromagnets.  In economics and social sciences, many realistic models must recognize that individual agents are influenced by the actions of others \cite{carmona2013control,tembine2011mean}.  Analogous statements hold true for various systems in biology \cite{morale2005interacting,shimizu1974muscular, shimizu1972phenomenological, touboul2011mean,zhu2011hybrid}.  

In many cases, these interactions are mediated by one or more mean fields - that is, ensemble averages over all the particles.  A quite general form for this type of interaction is captured in the system of SDEs
\begin{equation} \label{finiteN}
	dX^i_t = \alpha \left( X^i_t, t, \frac{1}{N_p} \sum_{j=1}^{N_p} R(X^j_t) \right) dt + \beta \left( X^i_t, t, \frac{1}{N_p} \sum_{j=1}^{N_p} R(X^j_t) \right) dW^i_t.
\end{equation}
Here, $X^i_t \in \R^d$ is the $i^\textrm{th}$ particle trajectory at time $t$, the $W^i_t$ are independent $D$-dimensional standard Brownian motions, $N_p$ is the total number of particles, $R: \R^d \rightarrow \R^\gamma$ governs the mean-field interaction, $\alpha: \R^d \times \R^\gamma \rightarrow \R^d$, and $\beta: \R^d \times \R^\gamma \rightarrow \R^{d \times D}$.  

For any fixed $N_p$, this is a system of $N_pd$ coupled SDEs for the $X^i_t$, $i = 1, 2, ..., N_p$.  In applications, particularly those arising from physics, $N_p$ is exceedingly large - $N_p \approx 10^{23}$ is not unusual.  This makes storage and evolution of even a single sample of the $N_pd$ dimensional system state computationally intractable, much less the many samples required for a Monte Carlo simulation.  Intuitively, one would like to view each of the $X^i_t$ as a $d$-dimensional sample from a single overarching stochastic process, thereby dramatically reducing the dimensionality of the problem\footnote{This is closely related to the ``molecular chaos" assumption used in deriving the Botlzmann equation from the BBGKY hierarchy.}.  

The mathematical justification for this intuition was developed by McKean \cite{mckean1967propagation,mckean1966class}, who showed that in the limit $N_p \rightarrow \infty$, the evolution of each $X^i_t$ converges to
\begin{equation} \label{MV}
	dX_t = \alpha \left( X_t, t, \E R(X_t) \right) dt + \beta \left( X_t, t, \E R(X_t) \right) dW_t
\end{equation}
in the strong sense\footnote{This is a special case of McKean's result.  In general, the pathwise and mean field dependencies need not be factorable.  For example, $dX = \E_{X'} \left[ (1 + (X - X')^2)^{-1} \right] \, dt + dW$, where $X'$ is iid relative to $X$, is a McKean-Vlasov equation to which McKean's theorem applies, but cannot be written in the form (\ref{MV}). However, the factorization greatly simplifies application of MLMC and is present in many applications.} at rate $N_p^{-1/2}$, subject to reasonable assumptions on $\alpha$, $\beta$, and $R$.  It is important to note that $X^i_t$ and $X^j_t$ satisfying (\ref{MV}) are independent for $i\neq j$ when independent Brownian motions are used, while this is not the case for distinct particles in (\ref{finiteN}).  This result and others of a similar character have come to be called ``propagation of chaos" theorems - see \cite{sznitman1991topics} for a review.  

The process (\ref{MV}) is an instance of a McKean-Vlasov process - a term which has come to encompass any SDE whose coefficients depend on the probability density of the solution as well as its path-wise behavior.  In the same way that the probability density of an SDE's solution solves the forward Kolmogorov (i.e.\ Fokker-Planck) equation, the probability density $p(x,t)$ corresponding to (\ref{MV}) solves the following nonlinear, nonlocal PDE \cite{mckean1966class}:
\begin{equation} \label{MVPDE}
	\partial_t p + \nabla \cdot \left\{ \alpha \left( x, t, \bar{R} \right) p \right\} = \frac{1}{2} \frac{\partial^2}{\partial x_i \partial x_k} \left\{ \beta_{ij}(x,t,\bar{R}) \beta_{kj}(x,t,\bar{R}) p \right\},
\end{equation}
where
\begin{equation} \label{MFforPDE}
	\bar{R} \coloneqq \int_{\R^d} R(y) p(y,t) \, dy
\end{equation}
and summation over repeated indices is implied. 

Any Monte Carlo method for (\ref{MV}) is thus expected also to be a valuable numerical approach for PDEs of type (\ref{MVPDE})-(\ref{MFforPDE}) in high dimension.  Examples include the Vlasov and Landau-Fokker-Planck equations governing kinetic plasma dynamics, and Fokker-Planck approximations of the Boltzmann equation, each of which have $d=6$.

Our new multilevel method extends the computational gains MLMC affords for SDEs to McKean-Vlasov processes of type (\ref{MV}).  In particular, only a logarithmic increase in complexity is seen relative to standard MLMC.  In contrast to previous methods, the approach requires coupling each level in the scheme to all lower levels.  This enables low variance estimates of the mean field at high levels, even when few particles are sampled.  This coupling does complicate analysis of the scheme, since distinct samples and levels are not independent.  This is overcome through use of propagation of chaos techniques.  Some subtleties in implementation are also introduced, but these only necessitate small algorithmic changes.  

The remainder of the paper is structured as follows.  Section 2 reviews the requisite background material, including standard MLMC for SDEs and single-level schemes for McKean-Vlasov processes.  Section 3 describes the motivation and intuition behind the multilevel scheme for McKean-Vlasov processes, then outlines the algorithm.  Section 4 develops a partial theory of the complexity and convergence of the algorithm.  Section 5 presents numerical examples, in which convergence is observed in settings more general than our theory requires.  We conclude in section 6.  Some of the more technical proofs are confined to appendices.  


\section{Background}
\subsection{MLMC for SDEs}
In its original form \cite{giles2008multilevel}, MLMC is a numerical technique for solving the following problem: If $X_t \in \R^d$ satisfies (\ref{SDE}) with $X_0$ known, estimate 
\begin{equation}
	\bar{P} \coloneqq \E \left[ P(X_t) \right].
\end{equation}
Here, $P: \R^d \times [0,T] \rightarrow \R$ is some pre-specified functional of the solution $X_t$.  In finance, $P$ is frequently called the `payoff function' - we will use this terminology here, even when discussing other applications.

MLMC improves on the standard Monte Carlo approach, which is to discretize in time, usually with the Euler-Maruyama scheme 
\begin{equation}
	X_{n+1} = X_n + a(X_n,t_n) \Delta t + b(X_n,t_n)\Delta W_n,
\end{equation}
where $X_n \approx X_{t_n}$, $t_n = n\Delta t$, and the $\Delta W_n$ are independent $N(0,\Delta t)$ random variables.  It is well known that this scheme has weak order 1 and strong order 1/2 \cite{kloeden1992numerical}.  One generates $N$ independent samples of the solution - indexed by $i$ - and estimates
\begin{equation}
	\bar{P} \approx \frac{1}{N} \sum_{s=1}^N P\left( X_n^i \right).
\end{equation}

This method - we will call it `standard Monte Carlo' - has time-stepping error $O(\Delta t)$ and sampling error $O(1/\sqrt{N})$.  The computational complexity $\kappa$ is proportional to the total number of time steps taken, and thus scales as
\begin{equation}
	\kappa \sim \frac{N}{\Delta t} = O \left( \varepsilon^{-3} \right),
\end{equation}
where $\varepsilon$ is the root-mean-square error in the estimate of $\bar{P}$.  In general, a scheme with weak order $q$ will have $\kappa \sim \varepsilon^{-(2+1/q)}$.  

MLMC improves upon this using an iterated control variate strategy.  In the simplest version, one introduces a hierarchy of time-steps $\Delta t_\ell = \Delta t_0 2^{-\ell}$, for $\ell = 0, 1, ..., L$, and uses the level $\ell-1$ process as a control variate for the level $\ell$ process.  More concretely, if $X^\ell$ is the process at level $\ell$, one wishes to approximate $\E [P(X^L)]$.  The identity
\begin{equation} \label{telescope}
	\E \left[ P\left(X^L\right) \right] = \E \left[ P\left(X^0\right) \right] + \sum_{\ell=1}^L \E \left[ P\left(X^\ell\right) - P\left(X^{\ell-1}\right) \right]
\end{equation}
holds trivially.  Denoting the variance of the $\ell^\textrm{th}$ term on the right side by $V_\ell$, strong convergence at rate $r$ implies that $V_\ell \sim \Delta t_\ell^{2r}$ when $X^\ell$ and $X^{\ell-1}$ are sampled using the same underlying Brownian path.  

One can derive the number number of samples necessary to achieve the desired accuracy $\varepsilon$ with minimal complexity.  We quote the result without proof, and refer the reader interested in more depth to \cite{giles2008multilevel}:
\begin{equation} \label{optnumsamps}
	N_\ell = \left\lceil \frac{2}{\varepsilon^2} \sqrt{V_\ell \Delta t_\ell} \sum_{m=0}^L \sqrt{\frac{V_m}{\Delta t_m}} \right\rceil.
\end{equation}
This choice of $N_\ell$ bounds the mean squared sampling error by $\varepsilon^2/2$.  One then begins with $L=1$ and increments it until
\begin{equation}
	\norm{ \frac{1}{N_L} \sum_{i=1}^{N_L} \left[ P\left(X^{L,i}\right) - P\left(X^{L-1,i}\right) \right] }^2 \leq \frac{\varepsilon^2}{2},
\end{equation}
thereby placing the same (approximate) bound on the mean-squared time-stepping error\footnote{This is a simpler criterion than the more conservative approach used in \cite{giles2008multilevel}.  Our scheme is not substantially changed by this choice, and we find this simpler criterion sufficient for convergence in all our numerical experiments.}.    

The total complexity of the method is
\begin{equation} \label{complexity}
	\kappa \sim \sum_{\ell=0}^L \frac{N_\ell}{\Delta t_\ell} \cong \frac{2}{\varepsilon^2} \left( \sum_{m=0}^L \sqrt{\frac{V_m}{\Delta t_m}} \right)^2 \sim \left\{
	\begin{array}{lr}
		\varepsilon^{-2} (\log \varepsilon)^2 & r = 1/2 \\ 
		\varepsilon^{-2} & r > 1/2
	\end{array}
	\right.
\end{equation}
Thus, MLMC scales better than standard Monte Carlo no matter what discretization is used for each, so long as a strong converge rate of at least 1/2 is maintained.  

\subsection{Existing Monte Carlo schemes for McKean-Vlasov processes}
Numerical solution of (\ref{MV}) is complicated by the fact that we do not have advance knowledge of $\E [R]$.  We are thus forced to approximate it by a sample mean - i.e. to sample $N$ particles and evolve according to
\begin{equation} \label{finiteNagain}
	dX^i_t =  \alpha \left( X^i_t, t, \frac{1}{N} \sum_{j=1}^{N} R(X^j_t) \right) dt + \beta \left( X^i_t, t, \frac{1}{N} \sum_{j=1}^{N} R(X^j_t) \right) dW^i_t.
\end{equation}
Even though (\ref{finiteNagain}) is formally identical to (\ref{finiteN}) with $N_p$ replaced by $N$, McKean's results ensure we are justified in regarding each $X_t^i$ as an approximate sample from the $d$-dimensional probability density corresponding to (\ref{MV}), rather than the vector of all the $X_t^i$ being a single sample of an $Nd$-dimensional random variable.  As such, we may regard 
\begin{equation}
	\widehat{P} \coloneqq \frac{1}{N} \sum_{i=1}^N P\left( X^i_t \right)
\end{equation}
as an estimate of $\E [P(X_t)]$ under the evolution (\ref{MV}).  

Of course, we generally must introduce a time-discretization for (\ref{finiteNagain}), so in practice we evolve
\begin{equation} \label{discreteSLMV}
	X^i_{n+1} = X^i_n + \alpha \left( X^i_n, t_n, \frac{1}{N} \sum_{j=1}^{N} R(X^j_n) \right) \Delta t + \beta \left( X^i_n, t_n, \frac{1}{N} \sum_{j=1}^{N} R(X^j_n) \right) \Delta W^i_n.  
\end{equation}
This scheme was introduced and partially analyzed in \cite{ogawa1992monte,ogawa1994monte}.  This work concerned strong convergence, and we will make use of its results.  In \cite{bossy1997stochastic,bossy1996convergence}, weak convergence was discussed, and \cite{antonelli2002rate} builds on that work to show that the cumulative distribution function for samples of (\ref{discreteSLMV}) approximates that of (\ref{MV}) to order $\Delta t + 1/\sqrt{N}$ in the $L^1$ norm.  This result, as well as those in \cite{bossy1997stochastic,bossy1996convergence}, is limited to $d=1$, where it roughly corresponds to a weak convergence rate.  In arbitrary dimension, the overall complexity of the algorithm is either $O(\varepsilon^{-3})$ - if $O(\Delta t)$ weak convergence holds generally - or $O(\varepsilon^{-4})$ if not.  

The dimensionality restriction in existing weak convergence results means they are not useful here.  However, for MLMC, the weak convergence rate is of little importance to the scheme's efficiency. Indeed, (\ref{complexity}) shows that improved weak convergence rates may only affect the constants in the complexity, not the scaling itself.  As a result, we focus only on strong convergence rate, which implies weak convergence at the same rate.  

\section{Description of Multilevel Algorithm}
The central result of this paper is a multilevel Monte Carlo algorithm for fast solution of the following problem: given that $X_t \in \R^d$ satisfies
\begin{equation}
	dX_t = a(X_t,t,\E R(X_t)) dt + b(X_t,t,\E R(X_t)) dW_t, \qquad X_0 = \xi
\end{equation}
for $t \in [0,T]$, with $\xi$ a random variable with some known distribution, estimate 
\begin{equation}
	\bar{P}(t) \coloneqq \E \left[ P(X_t) \right].
\end{equation}
As before, $P: \R^d \rightarrow \R^\eta$ is a pre-specified functional.  In the following, we will omit the explicit $t$ dependence of $\alpha$ and $\beta$ for brevity.

Aside from the presence of the mean field, note that we have changed our notion of payoff function to allow vector-valued output.  This does not complicate the analysis, but is useful for applications.  When discussing variances of vector-valued functions, we will denote
\begin{equation}
	\textrm{Var}[P(X)] = \max_{k} \textrm{Var} \left[ P_k(X) \right].
\end{equation}
In this way, we ensure that the required error tolerance is met in every component of $P$.  

\subsection{Algorithm Description}
We begin by establishing some notation.  Let $\Delta t_\ell = \Delta t_0 2^{-\ell}$, where $\ell = 0, 1, ..., L$.  $X^{\ell,i}_n$, the $i^\textrm{th}$ sample of the approximate solution using time-step $\Delta t_\ell$ at time $t^\ell_n \coloneqq n \Delta t_\ell$.  The multilevel approximations of $\E [R]$ and $\E [P]$ at level $\ell$, time $t^\ell_n$, $\widehat{R}^\ell_n$ and $\widehat{P}^\ell_n$, respectively.  We abbreviate $R(X^{\ell,i}_n)$ to $R^{\ell,i}_n$ and similarly for $P^{\ell,i}_n$.  We extend these definitions to non-integer multiples of $\Delta t_\ell$ by linear interpolation.  That is, 
\begin{equation}
	\widehat{R}^\ell_s = \left( s  - \lfloor s \rfloor \right) \widehat{R}^\ell_{\lceil s \rceil} + \left( 1 - s + \lfloor s \rfloor \right) \widehat{R}^\ell_{\lfloor s \rfloor}
\end{equation}
for any $s \in [0,T/\Delta t_\ell]$, and similarly for $\widehat{P}^\ell_s$, $R^{\ell,i}_s$, and so forth.

Because in this method we require estimates of $\E[R]$ at every time step of every level, writing a direct analogue of the telescoping sum (\ref{telescope}) is notationally cumbersome.  Instead, we describe an iterative procedure for moving from level $\ell-1$ to $\ell$.  We must begin at $\ell=0$, where we use the standard single level scheme (\ref{discreteSLMV}) with $N_0$ samples.  In our notation, 
\begin{equation} \label{lev0}
\begin{split}
	X^{0,i}_{n+1} &= X^{0,i}_n + \alpha \left( X^{0,i}_n, \widehat{R}^0_n \right) \Delta t_0 + \beta \left( X^{0,i}_n, \widehat{R}^0_n \right) \Delta W^{0,i}_n, \\
	\widehat{R}^0_n &= \frac{1}{N_0} \sum_{j=1}^{N_0} R^{0,j}_n.
\end{split}
\end{equation}
Our level zero estimate of $\E [P]$ is just
\begin{equation} \label{Plev0}
	\widehat{P}^0_n = \frac{1}{N_0} \sum_{j=1}^{N_0} P^{0,j}_n.
\end{equation}

To update our estimates from level $\ell-1$ to $\ell$, we set
\begin{equation} \label{update}
	\widehat{R}^\ell_n = \widehat{R}^{\ell-1}_{n/2} + \frac{1}{N_\ell} \sum_{j=1}^{N_\ell} \left[ R^{\ell,j}_n - \tilde{R}^{\ell,j}_{n/2} \right], \qquad \widehat{P}^\ell_n = \widehat{P}^{\ell-1}_{n/2} + \frac{1}{N_\ell} \sum_{j=1}^{N_\ell} \left[ P^{\ell,j}_n - \tilde{P}^{\ell,j}_{n/2} \right],
\end{equation}
where $\tilde{R}^{\ell,j}_n$ and $\tilde{P}^{\ell,j}_n$ are evaluations of $R$ and $P$ at a new process $\tilde{X}^{ell,j}_n$.  

The core insight of the algorithm is in the definition of $\tilde{X}^{\ell,j}_n$.  There are two requirements of this process.  First, the updated estimates in (\ref{update}) must have the correct expectations - $\E[\widehat{R}^\ell_n] = \E[R^{\ell,j}_n]$, and similarly for $\widehat{P}^\ell_n$ - in order for the method to be accurate.  Second, the correction terms - $(R^{\ell,j}_n - \tilde{R}^{\ell,j}_{n/2})$ and similarly for $P$ - must have small variance in order for the method to be efficient.  

To accomplish the first goal, the $\tilde{X}^{\ell,j}_n$ should be identically distributed to the $X^{\ell-1,j}_n$. To accomplish the second, they must be correlated with the $X^{\ell,j}_n$.  We achieve both using the following construction: $X^{\ell,i}_0 = \tilde{X}^{\ell,i}_0$ for all $i$, and 
\begin{equation} \label{levl}
\begin{split}
	X^{\ell,i}_{n+1} &= X^{\ell,i}_n + \alpha \left( X^{\ell,i}_n, \widehat{R}^\ell_n \right) \Delta t_\ell + \beta \left( X^{\ell,i}_n, \widehat{R}^\ell_n \right) \Delta W^{\ell,i}_n, \\
	\tilde{X}^{\ell,i}_{n+1} &= \tilde{X}^{\ell,i}_n + \alpha \left( \tilde{X}^{\ell,i}_n, \widehat{R}^{\ell-1}_n \right) \Delta t_{\ell-1} + \beta \left( \tilde{X}^{\ell,i}_n, \widehat{R}^{\ell-1}_n \right) \Delta \tilde{W}^{\ell,i}_n.
\end{split}
\end{equation}
Here, $\Delta \tilde{W}^{\ell,i}_n = \Delta W^{\ell,i}_{2n} + \Delta W^{\ell,i}_{2n+1}$ is the coarsened version of the level $\ell$ Brownian path, as in standard MLMC.

Since they are each evolved using time-step $\Delta t_{\ell-1}$ and mean field $\widehat{R}^{\ell-1}_n$, $\tilde{X}^{\ell,i}_n$ and $X^{\ell-1,i}_n$ are identically distributed, as required.  Further, use of the same initial data and underlying Brownian path for $X^{\ell,i}$ and $\tilde{X}^{\ell,i}$ leads one to expect they are well correlated, as in standard MLMC.  This expectation is confirmed in section 4.  

Note that none of the samples are independent, which complicates variance estimation.  In this case, 
\begin{equation}
\begin{split}
	\textrm{Var} \left[\widehat{P}^\ell_n \right] &= \textrm{Var} \left[\widehat{P}^{\ell-1}_{n/2} \right] + \frac{1}{N_\ell} \textrm{Var} \left[P^{\ell,i}_n - \tilde{P}^{\ell,i}_{n/2} \right] \\
	 &+ 2 \textrm{Cov} \left[ \widehat{P}^{\ell-1}_{n/2}, \left( P^{\ell,i}_n - \tilde{P}^{\ell,i}_{n/2}\right) \right] \\
	 &+ \left( \frac{N_\ell - 1}{N_\ell} \right)^2 \textrm{Cov} \left[ \left( P^{\ell,i}_n - \tilde{P}^{\ell,i}_{n/2}\right), \left( P^{\ell,j}_n - \tilde{P}^{\ell,j}_{n/2}\right) \right]
\end{split}
\end{equation}
and similarly for $\widehat{R}^\ell_n$.  Because of the correlation between $X^{\ell,i}_n$ and $\tilde{X}^{\ell,i}_{n/2}$, we expect the second variance on the right to be small.  In the absence of the mean field, the covariance terms are identically zero, but that is not the case here since the different samples and levels are coupled to each other through the $\widehat{R}$'s.  

Fortunately, though, propagation of chaos results strongly suggest that the distinct samples and levels tend toward independence as the number of samples grows large.  This is because the limiting process (\ref{MV}) \textit{does} have independent samples, and for large $N$ our process closely approximates this limit.  We therefore feel justified in approximating
\begin{equation}
	\max_n \textrm{Var} \left[\widehat{P}^L_n \right] \approx \sum_{\ell=0}^L \frac{V_\ell}{N_\ell},
\end{equation}
where $V_\ell \coloneqq \max_n \textrm{Var} \left[ P^{\ell,i}_n - \tilde{P}^{\ell,i}_{n/2} \right]$.  Note that for standard MLMC, this is an identity.  This intuition will be made precise in a restricted version of the mean field case by our convergence theory.    

\subsection{Algorithm Outline}
The algorithm roughly follows that outlined in \cite{giles2008multilevel}, with a few additional subtleties introduced by the mean field.  Let us briefly summarize the main differences and their origins.  

In executing standard MLMC, the following situation frequently arises: One has already computed $N_\ell^\textrm{old}$ samples at level $\ell$.  Then, after incrementing $L$, the sum in (\ref{optnumsamps}) grows, so that the total sample number needed at level $\ell$ is $N_\ell^\textrm{new} > N_\ell^\textrm{old}$.  For SDEs, this is not problematic.  All samples are independent, so one simply generates $(N_\ell^\textrm{new} - N_\ell^\textrm{old})$ additional samples and uses them to update the relevant quantities. 

However, in the current context, the new samples will change the value of $\widehat{R}^\ell_n$, so that we now have two sets of samples at level $\ell$ which use slightly different mean fields.  When we then endeavor to evolve $\tilde{X}^{\ell+1}_n$, it is unclear what value of the mean field to use in order to ensure that the $\tilde{X}^{\ell+1}_n$ and $X^\ell_n$ are identically distributed.  

Our algorithm attempts to avoid this difficulty by preventing $N_\ell$ from growing as $L$ increases.  We do this by tracking $L^\textrm{est}$, a prediction of the final value of $L$, and the variances $V_\ell$ we expect to encounter at higher levels that have yet to be sampled.  In this way, once we sample level $\ell$, we are unlikely to need to resample it. 

We present an outline of the algorithm first, then discuss these subtleties in more depth as they pertain to each step in the outline.  We define $\epsilon_\ell = \max_n \| \widehat{P}^\ell_n - \widehat{P}^{\ell-1}_{n/2} \|_{\infty}$, and the algorithm proceeds as follows:
\begin{Algorithm}[Mean-Field Multilevel Monte Carlo] \hspace{8em}
\begin{enumerate} 
	\item Fix an error tolerance $\varepsilon$ and set $L=1$.  
	\item Choose an initial time step $\Delta t_0$ and number of samples $N_0^i$ and $N_1^i$.  
	\item Compute $\widehat{P}^0_n$, $\widehat{R}^0_n$, and $V_0$ using (\ref{lev0})-(\ref{Plev0}) with the specified $\Delta t_0$ and $N_0^i$ samples.  Then compute the analogous quantities at level 1 using (\ref{update})-(\ref{levl}) and $N_1^i$ samples.
	\item Estimate the number of necessary levels $L^{\textrm{est}}$ using
	\begin{equation} \label{Lest0}
		L^{\textrm{est}} = \left\lceil 2\log_2 \left( \epsilon_1 / \varepsilon\right) + 2 \right\rceil.
	\end{equation}
	\item Compute $N_0$ and $N_1$ using (\ref{optnumsamps}), but with $L$ replaced by $L^\textrm{est}$, and assume
	\begin{equation} \label{Vassump}
		\frac{V_\ell}{\Delta t_\ell} = \frac{V_1}{\Delta t_1}
	\end{equation}
	for each $2 \leq \ell \leq L^\textrm{est}$.  
	\item If $N_0 > N_0^i$ or $N_1 > N_1^i$, repeat steps (3)-(5) with $N_0^i \rightarrow N_0$, $N_1^i \rightarrow N_1$.  
	\item While $\epsilon_L > \varepsilon(1 - 1/\sqrt{2})$, iterate the following:
	\begin{enumerate}
		\item Increment $L$.  
		\item Set $N_L = N_{L-1} / 2$.  
		\item Compute $\widehat{P}^L_n$, $\widehat{R}^L_n$, and $V_L$ using (\ref{update})-(\ref{levl}) with $N_L$ samples.  
		\item Set $L^\textrm{est}$ via 
		\begin{equation} \label{Lestupdate}
			L^{\textrm{est}} = L + 1 + \left\lceil 2\log_2 \left( \epsilon_L / \varepsilon \right) \right\rceil.
		\end{equation}¥
		\item Set $N_\ell$, $\ell = L, ..., L^\textrm{est}$ according to (\ref{optnumsamps}), again replacing $L$ with $L^\textrm{est}$ and assuming
		\begin{equation} \label{Vassumpupdate}
			\frac{V_\ell}{\Delta t_\ell} = \frac{V_L}{\Delta t_L}
		\end{equation}
		for all $\ell$ satisfying $L < \ell \leq L^\textrm{est}$.  
	\end{enumerate}
	\item Return $\widehat{P}^L_n$.  
\end{enumerate}
\end{Algorithm}

The first three steps in the outline above are standard.  In step 4, we estimate the total number of levels we expect to use, based only on knowledge of $\ell = 0, 1$.  The estimates come from our strong convergence results, which imply
\begin{equation}
	\left\| \E P(X^\ell_n) - \E P(X_{t^\ell_n}) \right\| \approx c \Delta t_\ell^{1/2} 
\end{equation}
for some unknown $c$.  Using what essentially amounts to Richardson extrapolation, it is straightforward to show that the smallest $L$ satisfying
\begin{equation}
	\left\| \E P(X^\ell_n) - \E P(X_{t^\ell_n}) \right\| \leq \frac{\varepsilon}{\sqrt{2}}
\end{equation}
is approximately given by (\ref{Lest0}).  A directly analogous computation gives (\ref{Lestupdate}).  

Step 7b) and equations (\ref{Vassump}) and (\ref{Vassumpupdate}) arise directly from our expectation that $V_\ell$ scales like $\Delta t_\ell$ - this is confirmed (up to a logarithmic factor) in the following section.  Since $N_\ell \propto \sqrt{V_\ell \Delta t_\ell}$, we expect each $N_\ell$ to be half the size of the previous one, thus justifying 7b).  Moreover, one expects each term in the sum appearing in (\ref{optnumsamps}) to be equal, so we use the last known term to estimate the remaining unknown terms.  

Inutitively, then, this algorithm yields approximations $\widehat{P}^L_n$ of $\E P(X_{t^L_n})$ satisfying
\begin{equation}
	\max_n \left( \E \left[ \left( \widehat{P}^L_n - \E P(X_{t^L_n}) \right)^2 \right] \right)^{1/2} \approx \varepsilon
\end{equation}
with computational complexity comparable to that of standard MLMC methods.  In the following section, we will see that, in fact, the complexity is increased by only a logarithmic factor as a result of the mean field.  

\section{Convergence and Complexity Results}

\subsection{The Linear Case}

We prove convergence of the multilevel scheme for McKean-Vlasov processes of the form
\begin{equation} \label{linear}
	dX_t = \left(AX_t + B\E\left[ X_t \right] \right)dt + \sigma(t) dW_t,
\end{equation}
where $A$ and $B$ constant matrices.  For $d=1$ and certain choices of $A$ and $B$, this is the Shimizu-Yamada model of muscle contraction \cite{frank2005nonlinear,shimizu1974muscular,shimizu1972phenomenological}.  A model of this type has also been used to describe target leverage ratios in finance \cite{lo2012simple}.  

The $X^{\ell,i}_n$ satisfy
\begin{equation}
	X^{\ell,i}_{n+1} = X^{\ell,i}_n + \left(A X^{\ell,i}_n + B \widehat{X}^\ell_n \right) \Delta t_\ell + \sigma(t^\ell_n) \Delta W^{\ell,i}_n,
\end{equation}
and 
\begin{equation}
	\widehat{X}^\ell_n = \widehat{X}^{\ell-1}_{n/2} + \frac{1}{N_\ell} \sum_{i=1}^{N_\ell} \left[ X^{\ell,i}_n - \tilde{X}^{\ell,i}_{n/2} \right].
\end{equation}  

\subsubsection{Strong Convergence Theory}
Following standard propagation of chaos arguments, we define the new quantity $\mathcal{X}^{\ell,i}_n$ by
\begin{equation}
	\mathcal{X}^{\ell,i}_{n+1} = \mathcal{X}^{\ell,i}_n + (A \mathcal{X}^{\ell,i}_n + B \E [\mathcal{X}^{\ell,i}_n]) \Delta t_\ell + \sigma(t^\ell_n) \Delta W^{\ell,i}_n.
\end{equation}
Simulation of $\mathcal{X}^\ell_n$ is impossible without advance knowledge of its expectation, but it has the important feature that its samples are independent.  Intuitively, one expects that $X^{\ell,i}_n \rightarrow \mathcal{X}^{\ell,i}_n$ as the $N_\ell \rightarrow \infty$.  The following theorem confirms this intuition and gives the rate of convergence:
\begin{thrm}[Convergence]
Algorithm 1 applied to (\ref{linear}) has the following bound on each sample:
\begin{equation}
	\max_{r\leq \ell} \max_{n \leq T / \Delta t_r} \E \norm{X^r_n - \mathcal{X}^r_n} \leq K' \left( \frac{1}{\sqrt{N_0}} + \sum_{m=1}^\ell \sqrt{\frac{\Delta t_m}{N_m}} \right)
\end{equation}
for some positive constant $K'$ that is independent of $\ell$.  
\end{thrm}

\begin{proof}
See appendix A.
\end{proof}

Standard techniques show that the $\mathcal{X}^{\ell,i}_n$ converge strongly to solutions of (\ref{linear}) at rate $1/2 - \delta$ for any positive $\delta$ - see e.g.\ proposition 3 in \cite{ogawa1992monte}.  As an immediate result, we have the following corollary
\begin{coro}
The $X^{\ell,i}_n$ converge strongly to solutions of (\ref{linear}).  In particular, if $X_t^i$ is a solution of (\ref{linear}) using the same initial data and Brownian path as $X^{\ell,i}_n$, then for any $\delta > 0$
\begin{equation}
	\E \norm{ X^{\ell,i}_n - X^i_{t^\ell_n} } \leq K \left( \Delta t_\ell^{1/2-\delta} + \frac{1}{\sqrt{N_0}} + \sum_{m=1}^\ell \sqrt{\frac{\Delta t_m}{N_m}} \right)
\end{equation}
for some $K$ that is independent of $\ell$.  
\end{coro}
This completes our understanding of the strong convergence of the multilevel samples in this special case.  This implies  weak convergence at the same rate.  


\subsubsection{Complexity Theory}
An additional corollary to Theorem 1 which leads directly into our complexity analysis is as follows:
\begin{coro}
For any Lipschitz payoff function $P$ and $\delta > 0$, 
\begin{equation}
	\E \norm{\widehat{P}^\ell_n - \E P(X_{t^\ell_n}) } \leq K \left( \Delta t_\ell^{1/2-\delta} + \frac{\ell}{\sqrt{N_0}} + \sum_{m=1}^\ell m \sqrt{\frac{\Delta t_m}{N_m}} \right)
\end{equation}
for some $K$ independent of $\ell$.  
\end{coro}
\begin{proof}
We first define 
\begin{equation}
	\widehat{\mathcal{P}}^\ell_n = \widehat{\mathcal{P}}^{\ell-1}_{n/2} + \sum_{i=1}^{N_\ell} \left[ P(\mathcal{X}^{\ell,i}_n) - P(\tilde{\mathcal{X}}^{\ell,i}_{n/2}) \right],
\end{equation}
where the definition of $\tilde{\mathcal{X}}^{\ell,i}_n$ is analogous to that of $\tilde{X}^{\ell,i}_n$.  We have  
\begin{equation}
	\E \norm{\widehat{P}^\ell_n - \E P(X_{t^\ell_n}) } \leq \E \norm{ \widehat{P}^\ell_n - \widehat{\mathcal{P}}^\ell_n} + \E \norm{ \widehat{\mathcal{P}}^\ell_n - \E \widehat{\mathcal{P}}^\ell_n } + \norm{ \E \widehat{\mathcal{P}}^\ell_n - \E P(X_{t^\ell_n})}.  
\end{equation}
We bound each of the three terms on the right in turn.  First, Theorem 1 and summation up to $\ell$ immediately imply
\begin{equation}
	 \E \norm{ \widehat{P}^\ell_n - \widehat{\mathcal{P}}^\ell_n} \leq K \left( \frac{\ell}{\sqrt{N_0}} + \sum_{m=1}^\ell m \sqrt{\frac{\Delta t_m}{N_m}} \right).
\end{equation}
Since $\E \widehat{\mathcal{P}}^\ell_n = \E P(\mathcal{X}^\ell_n)$ by construction, the independence and strong convergence of the $\mathcal{X}^{\ell,i}_n$ imply
\begin{equation}
	\E \norm{ \widehat{\mathcal{P}}^\ell_n - \E \widehat{\mathcal{P}}^\ell_n } \leq K \left( \frac{1}{N_0} + \sum_{m=1}^\ell \frac{\Delta t_m}{N_m} \right)^{1/2}, \qquad \norm{ \E \widehat{\mathcal{P}}^\ell_n - \E P(X_{t^\ell_n})} \leq K \sqrt{\Delta t_\ell}.
\end{equation}
Combining these three bounds and using the sub-additivity of the square root function gives the desired result.  
\end{proof}

\begin{thrm}[Complexity]
For the MLMC algorithm (\ref{lev0})-(\ref{Plevl}) applied to (\ref{linear}), there exist choices of $L$ and $N_\ell$ depending on $\varepsilon$ such that 
\begin{equation} \label{accbnd}
	\E \norm{\widehat{P}^L_n - \E P(X_{t^L_n}) } \leq \varepsilon
\end{equation}
and the computational complexity $\kappa$ satisfies
\begin{equation}
	\kappa = O \left( \varepsilon^{-2} |\log \varepsilon|^5 \right).
\end{equation}
\end{thrm}
\begin{proof}
Recalling Corollary 2, we simply choose $L$ and $N_\ell$ so that the time-stepping and sampling errors are each bounded by $\varepsilon / 2$.  Namely, 
\begin{equation}
\begin{split}
	L &= \frac{2}{1-\delta} \log_2 \left( 2K\sqrt{\Delta t_0} / \varepsilon \right),\\
	N_0 &= \left\lceil 4K L^2 (L+1)^2 \varepsilon^{-2} \right\rceil, \\
	N_\ell &= \left\lceil 4K \Delta t_\ell L^2 (L+1)^2 \varepsilon^{-2} \right\rceil \quad \textrm{for } \ell \geq 1.
\end{split}
\end{equation}
This clearly implies (\ref{accbnd}).  Moreover, 
\begin{equation}
	\kappa \sim \sum_{\ell=0}^L \frac{N_\ell}{\Delta t_\ell} \sim \varepsilon^{-2} L^3 (L+1)^2 \sim \varepsilon^{-2} | \log \varepsilon |^5,
\end{equation}
thus completing the proof.  
\end{proof}
Evidently, the presence of the mean field only increases the complexity of MLMC by an additional logarithmic factor.  It bears noting that this proof is not constructive - without advance knowledge of $K$, it gives no information about how to choose $L$ and only tells us about the relative sizes of the $N_\ell$.  The choices specified in 3.2 are motivated by analogy to the standard MLMC method, and are found to be effective in numerical experiments.  


\subsection{Partial results toward a more general theory}

Of interest is the extension to more general classes of McKean-Vlasov processes than (\ref{linear}).  While a complete theory is work in progress, we present partial results.  Our numerical results provide strong evidence for the convergence and complexity of the scheme in quite general contexts.  

We begin by reprinting for clarity the type of process we work with:
\begin{equation} \label{McKeanVlasov}
	dX_t = \alpha\left(X_t,t,\E R(X_t) \right) dt + \beta\left(X_t,t,\E R(X_t) \right) dW_t, \qquad X_0 = \xi,
\end{equation}
where $\xi$ is a random variable with some known density.  A first desirable result is that $X^{\ell,i}_n$ and $\tilde{X}^{\ell,i}_{n/2}$ should be nearby in a strong sense, so that the variance of the level differences is small.  We work under the following mild assumptions:
\begin{assumption}
	Both $\alpha$ and $\beta$ have uniform Lipschitz bounds in each of their arguments.  
\end{assumption}
\begin{assumption}
	Both $\alpha$ and $\beta$ have uniformly bounded expectations.  That is, for any time interval $t \in [0,T]$ and $p \geq 1$, there exists a constant $K_{T,p}>0$ such that
	\begin{equation}
		\E \left[ \norm{ \alpha\left(X,t,\E R(X(t)) \right) }^p \right] \leq K_{T,p}
	\end{equation}
	and similarly for $\beta$, for every $p \in [1,p_{\textrm{max}}]$, where $p_\textrm{max} \geq 2$.  
\end{assumption}
We then have the following theorem, which demonstrates that at the time steps shared by levels $\ell$ and $\ell-1$, the approximations at the two levels are close in the strong sense.
\begin{thrm}
	Under assumptions 1 and 2, the multilevel scheme for (\ref{McKeanVlasov}) satisfies
	\begin{equation}
		\E \left[ \max_{k\leq n} \norm{X^{\ell}_k - \tilde{X}^\ell_{k/2} }^p \right] = \left\{
		\begin{array}{lcc}
		O\left(\sqrt{\Delta t_\ell}\right) & : & p=1 \\
		O\left( \Delta t_\ell \left| \log \Delta t_\ell \right| \right) & : & p=2 \\
		O\left(\Delta t_\ell \right) & : & 2<p \leq p_\textrm{max}
	\end{array} \right.
	\end{equation}
	for every even $n$ satisfying $n \leq T / \Delta t_{\ell}$, and with the maximum taken only over even $k$.  
\end{thrm}

\begin{proof}
See appendix B.
\end{proof}
This is the bulk of the desired result for even $n$.  The following corollary generalizes to odd $n$: 
\begin{coro}
Under assumptions 1 and 2, 
\begin{equation}
		\max_{k\leq n} \E \left[ \norm{ X^{\ell}_k - \tilde{X}^\ell_{k/2} }^p \right] = \left\{
		\begin{array}{lcc}
		O\left(\sqrt{\Delta t_\ell}\right) & : & p=1 \\
		O\left( \Delta t_\ell \left| \log \Delta t_\ell \right| \right) & : & p=2 \\
		O\left(\Delta t_\ell \right) & : & 2<p \leq p_\textrm{max}
	\end{array} \right.
	\end{equation}
for every $n$ (even or odd) satisfying $n \leq T/\Delta t_{\ell}$.  
\end{coro}

\begin{proof}
For even $n$, this is a direct consequence of theorem 1.  For odd $n$, we note that
\begin{equation}
\begin{split}
	\E \left[ \norm{ X_{n}^{\ell} - \tilde{X}_{n/2}^{\ell} }^p \right] &\leq
	K\E  \left[ \norm{ X_{n-1}^{\ell} - \tilde{X}_{(n-1)/2}^{\ell} }^p \right] \\
	&+ K\E \left[ \norm{\alpha_{n-1}^{\ell} - \tilde{\alpha}_{(n-1)/2}^{\ell} }^p \right] \Delta t_{\ell}^p \\
	&+ K\E \left[ \norm{\beta_{n-1}^{\ell} -  \tilde{\beta}_{(n-1)/2}^{\ell} }^p \right] \Delta t_{\ell}^{p/2}.
\end{split}
\end{equation}
Since $n$ is odd, $n-1$ is even and theorem 1 provides a bound for the first norm on the right.  The bounds on expectations of $\alpha$ and $\beta$ give uniform bounds on the second two norms.  The result immediately follows by taking maximums.   
\end{proof}
 
We thus have the following bound on the variance of the level differences.  
\begin{coro}
For any Lipschitz function $P$, 
\begin{equation}
	\max_{k\leq n} \textrm{Var} \left[ P^{\ell}_k - \tilde{P}^\ell_{k/2} \right] = O\left( \Delta t_\ell | \log \Delta t_\ell | \right)
\end{equation}
for every $n \leq T / \Delta t_{\ell}$.  
\end{coro}

\begin{proof}
By assumption 1 above, we have
\begin{equation}
\begin{split}
	\textrm{Var}\left[P^{\ell}_n - \tilde{P}^{\ell}_{n/2}\right] &\lesssim \E \left[ \norm{X_n^{\ell+1} - X_{n/2}^\ell}^2 \right] + \norm{\E\left[ P_n^{\ell} - \tilde{P}_{n/2}^\ell \right]}^2 \\
	&\lesssim \E \left[ \norm{X_n^{\ell} - \tilde{X}_{n/2}^\ell}^2 \right] + \E \left[ \norm{X_n^{\ell} - \tilde{X}_{n/2}^\ell} \right]^2 
\end{split}
\end{equation}
for $n$ even, while for odd $n$,
\begin{equation} \label{varOdd}
\begin{split}
	\textrm{Var}\left[P^{\ell}_n - \tilde{P}^{\ell}_{n/2}\right] &\lesssim
	 \E \left[ \norm{X_n^{\ell} - \tilde{X}_{n/2}^\ell}^2 \right] + \E \left[ \norm{\tilde{X}_{n/2}^{\ell} - \tilde{X}_{(n+1)/2}^\ell}^2 \right] +  \E \left[ \norm{\tilde{X}_{n/2}^{\ell} - \tilde{X}_{(n-1)/2}^\ell}^2 \right] \\
	 &+ \norm{\E\left[ P_n^{\ell} - \tilde{P}_{n/2}^\ell \right]}^2 + \norm{\E\left[ \tilde{P}_{n/2}^{\ell} - \tilde{P}_{(n+1)/2}^\ell \right]}^2 + \norm{\E\left[ \tilde{P}_{n/2}^{\ell} - \tilde{P}_{(n-1)/2}^\ell \right]}^2 \\
	 &\lesssim  \E \left[ \norm{X_n^{\ell+1} - X_{n/2}^\ell}^2 \right] + \E \left[ \norm{X_n^{\ell+1} - X_{n/2}^\ell} \right]^2 + O(\Delta t_\ell).
\end{split}
\end{equation}
By corollary 3, we have
\begin{equation}
	\textrm{Var}\left[P(X^{\ell+1}_n) - P^{\ell}_{n/2}\right] =
	O(\Delta t_\ell |\log \Delta t_\ell |) + O\left(\Delta t_\ell \right) 
\end{equation}
for both even and odd $n$.  Everything is independent of $n$, so maximizing over $n$ and ignoring the smaller terms gives the desired result.
\end{proof}

If $\kappa$ scales like 
\begin{equation}
	\kappa \sim \frac{2}{\varepsilon^2}\left( \sum_{\ell=0}^L \sqrt{\frac{V_\ell}{\Delta t_\ell}} \right)^2
\end{equation}
as in the standard MLMC case - see (\ref{complexity}) - this variance scaling yields an algorithm with $\kappa \sim \varepsilon^{-2} (\log \varepsilon)^4$, in close approximation to the result of Theorem 2.


\section{Numerical Results}

We conduct three types of numerical test.  First, we apply our scheme to an equation of type (\ref{linear}) to confirm the convergence and complexity results we've proved in that context.  Second, we consider a plane-rotator model of a ferromagnet with infinite interaction range.  Models of this type have been explored by a variety of authors \cite{frank2005nonlinear,kostur2002nonequilibrium,miyashita1980monte,miyashita1978monte}.  Even though this system does not satisfy the hypotheses of section 4.1, we observe the same convergence and complexity scalings.  Third, we apply our method to the Vlasov-Poisson system, which has many applications in plasma physics.  We again observe the predicted convergence and complexity behavior.  


\subsection{Linear equation tests}
We work in $d=1$ with fixed $\sigma$.  That is, we solve
\begin{equation}
	dX_t = (aX_t + b\E [X_t])dt + \sigma dW_t.  
\end{equation}
It is straightforward to find exact solutions for the mean and variance of the solution.  Specifically, by taking moments of the corresponding PDE and solving the resulting ODEs, we find
\begin{equation} \label{exsol}
	\E [X_t] = \E [X_0] e^{(a+b)t}, \qquad \textrm{Var}[X_t] = \left(\textrm{Var}[X_0] + \frac{\sigma^2}{2a}\right)e^{2at} - \frac{\sigma^2}{2a}.
\end{equation}
These exact results are used for comparison to the results of our multilevel method.  

We set $a = -1/2$, $b = 4/5$, $\sigma^2 = 1/2$, and $P(x) = x^2$.  Since $R(x) = x$ in this case, we can estimate the variance of $X_{t^L_n}$ by 
\begin{equation}
	\textrm{Var} \left[X_{t^L_n}\right] \approx \widehat{P}^L_n - \left( \widehat{R}^L_n \right)^2,
\end{equation}
and compare to the exact solution in (\ref{exsol}).  We first examine the convergence and complexity of the method.  These results appear in fig.\ \ref{LinSysConvergence}.  Here and in the proceeding convergence studies, error bars give estimates of the standard deviation of the $L^1$-error found by performing 20 independent computations for each $\varepsilon$.  Model complexities and computation times shown are averages over those same 20 computations. 

\begin{figure}[h] 
  \centering
	\includegraphics[width=.49\textwidth]{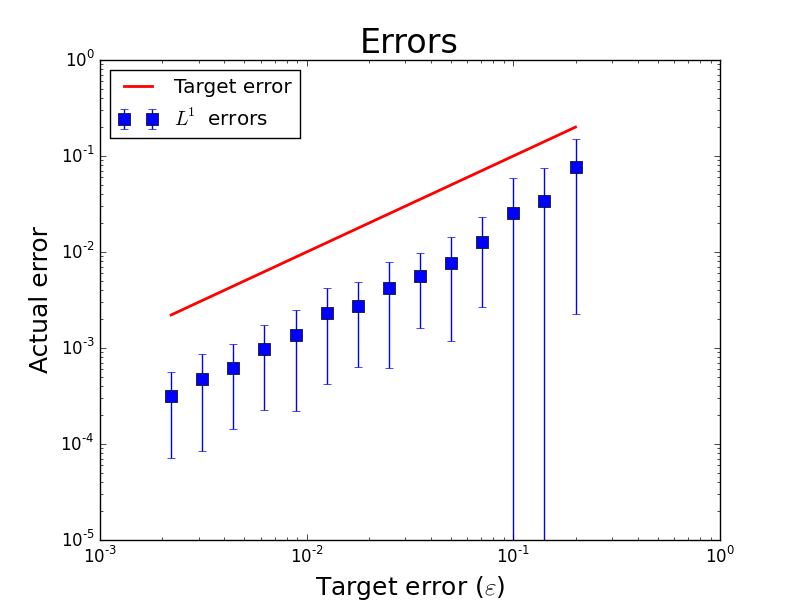}
	\includegraphics[width=.49\textwidth]{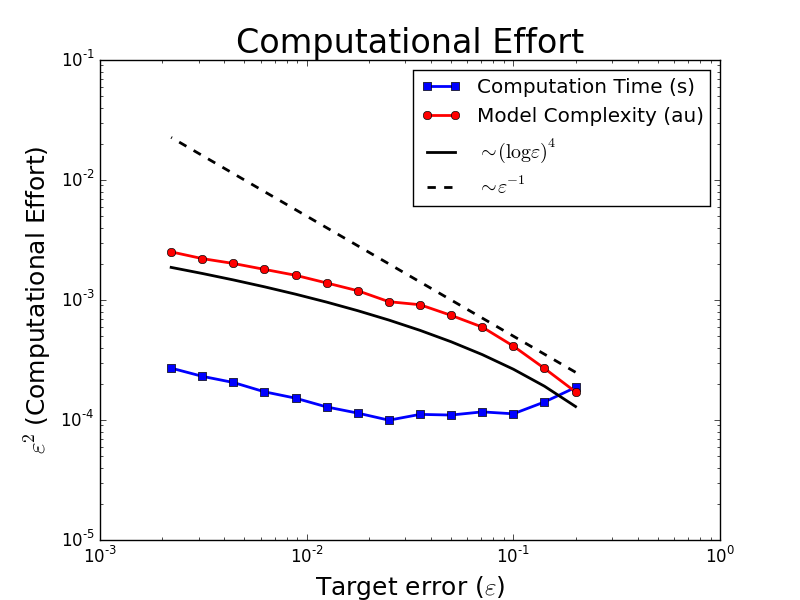}
	\caption{\label{LinSysConvergence} Convergence and complexity studies for the linear system, demonstrating the predicted behavior.}
\end{figure}

We then check that the $V_\ell$ scale as expected.  These results are shown in fig.\ \ref{LinSysVarScale}, where we find strong evidence for $O(\Delta t_\ell)$ scaling - slightly better than our theory requires.    

\begin{figure}[h] 
  \centering
	\includegraphics[width=.7\textwidth]{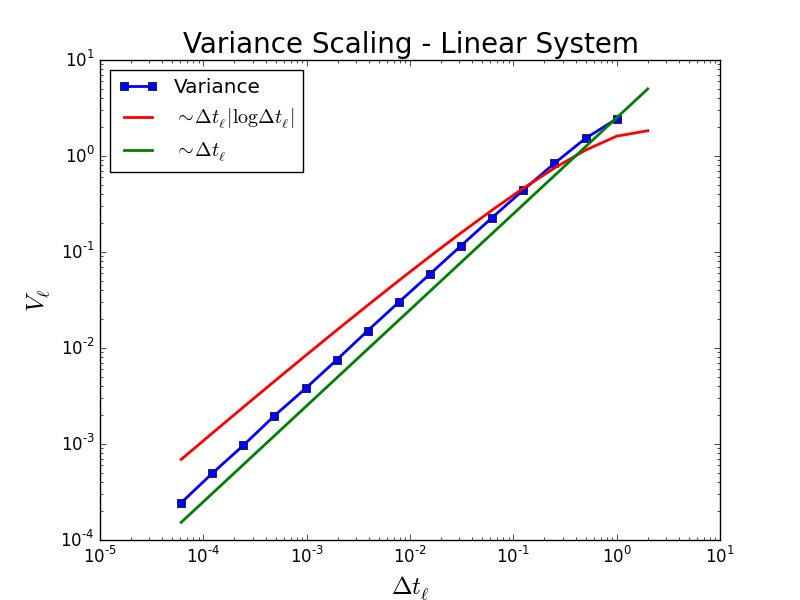}
	\caption{\label{LinSysVarScale} Variance of level differences in the linear model.  The observed $O(\Delta t_\ell)$ scaling is consistent with our analysis.}
\end{figure}
Additionally, it is worth noting that in our computations, both here and in our subsequent tests, we find evidence for $O(\Delta t)$ weak error, which is better than our strong convergence analysis requires.  This is manifested in the fact that we require an additional level when $\varepsilon$ is reduced by a factor of 2, not $\sqrt{2}$ as one expects for $O(\sqrt{\Delta t})$ weak convergence.  A rigorous proof of this improved weak convergence in arbitrary dimension is an interesting avenue for future research.  
\subsection{Plane-rotator tests}
We use the equation studied in \cite{kostur2002nonequilibrium}: 
\begin{equation}
	dX_t = \left\{ K \E' \left[ \sin (X_t' - X_t) \right] -\sin X_t \right\} \, dt + \sqrt{2\tau} dW_t,
\end{equation}
where $X_t \in \R$, $\tau$ is the background temperature, $K$ a coupling constant, and $\E '$ indicates expectation over the primed variable.  Physically, $X_t$ is the angle from some fixed axis of some oscillator - e.g.\ the magnetic moment of an atom or molecule.  These are presumed to interact via the $K \E'\sin(X_t' - X_t)$ term, and are subject to some external field or anisotropy aligned with the axis that leads to the $-\sin X_t$ term.  The stochastic term represents immersion in a heat bath at temperature $\tau$.  

A simple trigonometric identity makes our method directly applicable:
\begin{equation}
	dX_t = \left\{ K \left( \E \left[\sin X_t \right] \cos X_t - \E \left[ \cos X_t \right] \sin X_t \right) - \sin X_t \right\} \, dt + \sqrt{2\tau} dW_t.
\end{equation}
We set $K=1$, $\tau = 1/8$, $P(x) = \sin x$, $\Delta t_0 = T$, generate samples of $X_0$ from $N(\pi/2, 3\pi/4)$, and simulate to the terminal time $T=5$.  Our multilevel simulations are compared to the results of an over-resolved single level scheme of the type (\ref{discreteSLMV}), using $\Delta t = T/1024$ and $N = 5 \times 10^7$.

As above, we plot convergence and computation time data in fig.\ \ref{PRotConvComp}.  We again observe the expected behavior.   
\begin{figure}[h]
  \centering
	\includegraphics[width=.49\textwidth]{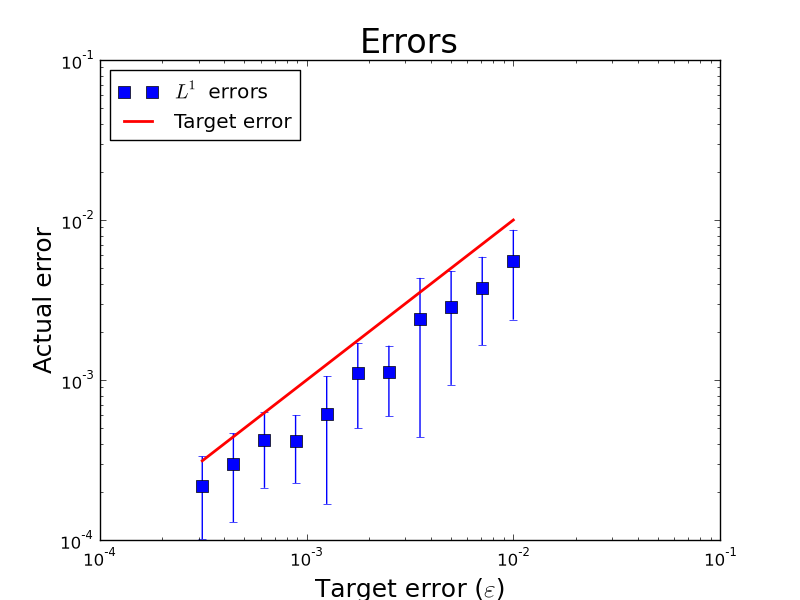}
	\includegraphics[width=.49\textwidth]{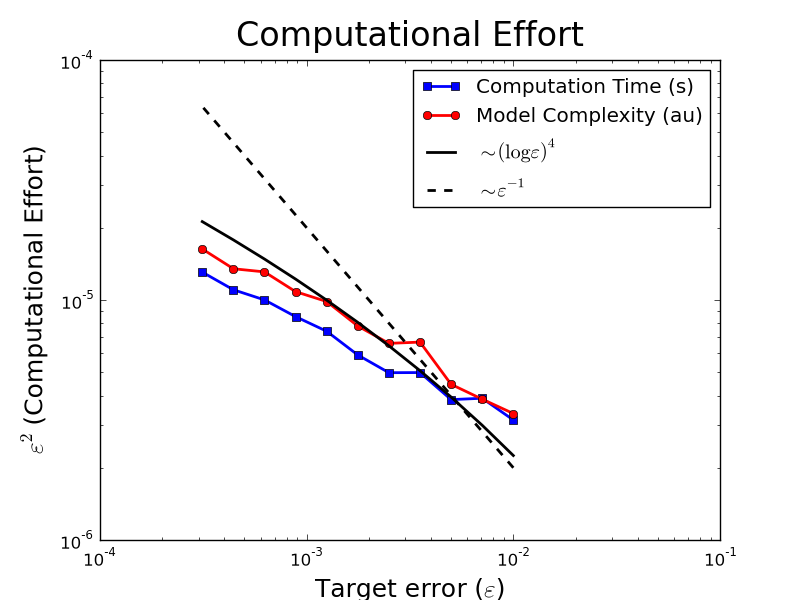}
	\caption{\label{PRotConvComp}(Left) Convergence of the multilevel method for the plane-rotator system.  (Right) Computational time and model complexity for plane-rotator system.}
\end{figure}
Additionally, we plot the $V_\ell$ in fig.\ \ref{PRotVarScale}.  We find results in agreement with the $O(\Delta t_\ell |\log \Delta t_\ell|)$ predicted by corollary 4.  
\begin{figure}[h] 
  \centering
	\includegraphics[width=.7\textwidth,height=.55\textwidth]{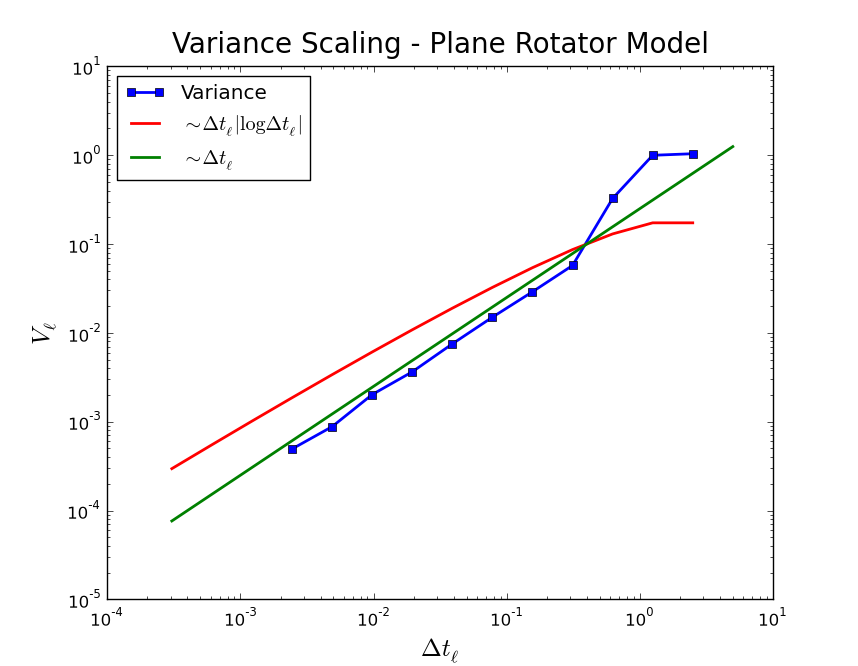}
	\caption{\label{PRotVarScale} Variance of level differences in in the plane rotator model.  The observed $O(\Delta t_\ell)$ scaling is consistent with our analysis.}
\end{figure}
\subsection{Particle-in-cell tests}
The Vlasov-Poisson system, 
\begin{equation} \label{VlasovPoisson}
	\partial_t f + v \cdot \nabla_x f - \nabla \phi \cdot \nabla_v f = 0, \qquad -\Delta \phi = \rho - \rho_0 = \int f \, dv - \rho_0
\end{equation}
for the phase space particle density $f(x,v,t)$ and electrostatic potential $\phi(x,t)$, is an important equation in a variety of plasma physics applications - see e.g.\ \cite{goldman2008vlasov, kulsrud1988analysis, sonnendrucker2001simulation,umeda2006vlasov,vay2002mesh}, among many others.  In the simplest case, $\rho_0$ is the constant $\rho_0 = \int \rho \, dx$.  

The particle-in-cell (PIC) method has long been the method of choice for this system.  A full review of the method is beyond the scope of this paper - we refer the interested reader to \cite{birdsall2004plasma}.  To briefly summarize, we rewrite (\ref{VlasovPoisson}) in terms of particle trajectories: 
\begin{equation}
	\frac{dx}{dt} = v, \qquad \frac{dv}{dt} = -\nabla \phi.
\end{equation}
For simplicity, we work in one space and one velocity dimension, with periodic boundary conditions.  We introduce a grid in $x$-space with grid points $x_i$ and grid size $h$.  We then approximate
\begin{equation}
	\rho(x_i) = \int \delta(x - x_i) f \, dx\, dv = \E_f \left[ \delta(x - x_i) \right] \approx \E_f \left[ S\left(\frac{x - x_i}{h} \right) \right],
\end{equation}
where $S(x)$ is some approximation of the Dirac delta function.  We use the most common choice in elementary PIC schemes: $S(x) = \max \left\{ 1 - |x|, 0 \right\}$.  

The $\rho(x_i)$ play the role of $R$, and the map $\rho(x_i) \rightarrow -\nabla \phi(x_i)$ - using the FFT to compute $\phi$ and its derivative - plays the role of $\alpha$.  We assume that $\nabla \phi$ is continuous and piecewise linear to completely specify $\alpha$.  

Evidently, $\beta \equiv 0$ for this system.  However, Coulomb collisions are typically introduced as a Fokker-Planck operator on the right side of (\ref{VlasovPoisson}), thus introducing non-zero $\beta$ and making use of the full weight of the work presented here.  This will be expanded upon in future work.    

In the numerical experiments presented here, we work in a dimensionless formulation with domain length $\mathcal{L} = 20$, $h = 1$, terminal time $T = 12$, and $\Delta t_0 = 1/3$. Initial particle positions are sampled from $N(10,6)\mod \mathcal{L}$, with the modulo present merely to ensure all particles reside in the computational domain.  Initial particle velocities are sampled from $N(0,1)$.  We let the payoff function be the vector of the $\rho(x_i)$.  That is, $P = R$ in this case.  

Convergence and computation time are shown in fig.\ \ref{PICconv}.  As before, we compare to a highly accurate single-level (standard PIC) simulation when the analytic solution is not available.  Here, we use $3.2 \times 10^5$ particles per cell and $\Delta t = 10^{-3}$ with $h$ unchanged.  We again observe the expected convergence and complexity scaling.  
\begin{figure}[h]
  \centering
	\includegraphics[width=.49\textwidth]{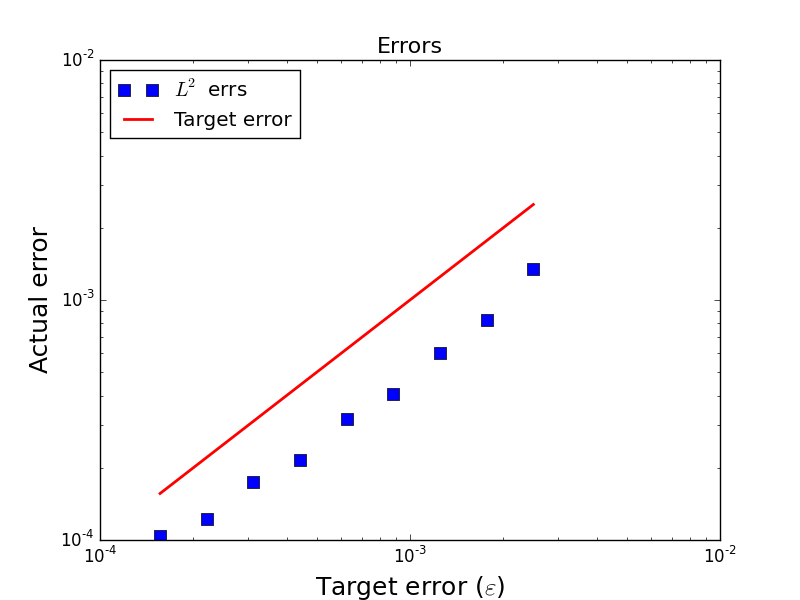}
	\includegraphics[width=.49\textwidth]{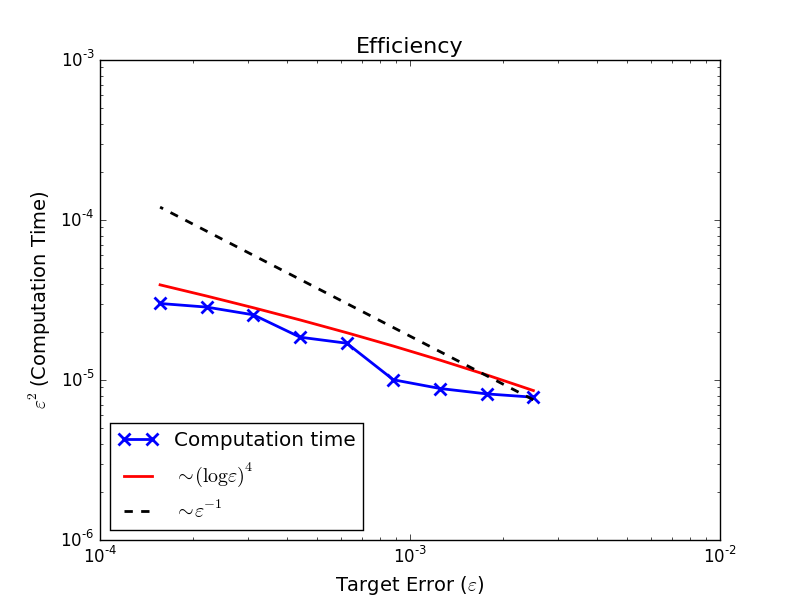}
	\caption{\label{PICconv}(Left) Convergence in density $\rho$.  (Right) Computation time as a function of $\varepsilon$, demonstrating the predicted scaling.}
\end{figure}
Admittedly, the factor of $\sim 4$ speed gain of the multilevel scheme in this context is not overly impressive.  This is largely due to the CFL condition's restriction on the size of $\Delta t_0$, which limits the number of levels one can explore using a serial code on a personal computer.  On a more powerful machine, one could explore smaller $\varepsilon$, where much larger speed increases are to be expected.  Moreover, an interesting direction for future research is the use of the multilevel scheme in concert with recent developments in implicit PIC schemes \cite{chacon2013charge,chen2013analytical,chen2011energy,chen2012efficient}, which eliminate upper bounds on time-step.  

\begin{figure}[h]
  \centering
	\includegraphics[width=.7\textwidth]{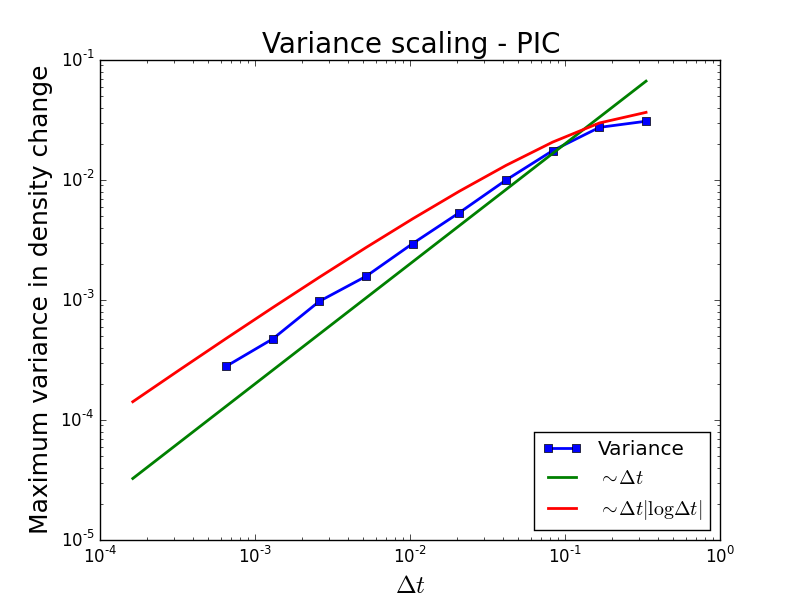}
	\caption{\label{PICVarScale}Variance of level differences in 1D1V multilevel PIC scheme.  The $O(\Delta t)$ scaling agrees with analysis.}
\end{figure}

As before, we also confirm that $V_\ell$ scales as predicted.  In fig.\ \ref{PICVarScale}, we again find excellent agreement with the prediction of corollary 4.


\section{Conclusions}
We have presented and analyzed a new multilevel Monte Carlo algorithm for simulation of McKean-Vlasov processes.  In a particular case, convergence and complexity results have been established that demonstrate performance comparable to the standard MLMC method for SDEs.  In addition, we have proved a variance scaling result that suggests similar performance in much greater generality.  

The extension of MLMC to McKean-Vlasov processes opens up a variety of applications to interacting particle systems.  Numerical tests have been performed that explore a few of those potential applications, including plane-rotator models and kinetic plasma dynamics.  In each case, the expected convergence and complexity are observed, even in cases in which our theory does not strictly apply.  

There are numerous directions in which this work could be extended.  From a theoretical perspective, rigorous convergence and complexity results in more generality are of obvious interest.  The results in section 4.2 represent a step in that direction, but are by no means complete.  HIgher-order time stepping schemes and their effect on the complexity and stability of the method are also areas for further investigation.  Additionally, in this work we have assumed that the refinement factor $\Delta t_\ell / \Delta t_{\ell-1} = 2$ for simplicity, but this choice of ratio has been found not to be optimal for standard MLMC \cite{giles2008multilevel} - for example, 4 is a popular choice.  It would be interesting to investigate generalization of the method presented here to arbitrary refinement factor, and to study what factor(s) might be optimal in this case.  

In applications, plasma simulation is of particular interest to the author.  As already mentioned, implicit time-stepping techniques used in conjunction with the multilevel approach presented here are particularly promising for accelerating PIC schemes, since this would eliminate the upper bound on $\Delta t_0$.  One also wonders whether multilevel methods in both space and time can be used in concert to further accelerate the simulation, as in \cite{haji2014multi}.  Each of these developments would also be a boon to simulations of Fokker-Planck models for rarefied gas dynamics - e.g.\ \cite{gorji2011fokker}.  Simulations in higher dimensions are of obvious interest, as are the incorporation of higher order time-stepping schemes and shape functions.  Each of these is a topic of current research for the author.


\section*{acknowledgements}
The author is particularly grateful to Mark Rosin, Russel Caflisch, and Jonathan Goodman for many, many fruitful conversations regarding this work.  Additional thanks go to Antoine Cerfon, Harold Weitzner, Bruce Cohen, Andris Dimits, and Jacob Bedrossian for discussions of plasma physical applications, as well as to Charles Newman and Robert Kohn.

\begin{appendix}

\section{Proof of Theorem 1}

For brevity, we will omit the sample index $i$ in this proof, since we always refer to the same sample.  We note first that
\begin{equation}
\begin{split}
	(X^\ell_n - \mathcal{X}^\ell_n) - (X^{\ell-1}_{n/2} - \mathcal{X}^{\ell-1}_{n/2}) &= \Delta t_\ell \sum_{k=0}^{n-1} \left\{ A (X^\ell_k - \mathcal{X}^\ell_k) + B(\widehat{X}^\ell_k - \E [\mathcal{X}^\ell_k]) \right\} \\
	&- \Delta t_{\ell-1} \sum_{k=0}^{n/2 - 1} \left\{ A (X^{\ell-1}_k - \mathcal{X}^{\ell-1}_k) + B(\widehat{X}^{\ell-1}_k - \E [\mathcal{X}^{\ell-1}_k]) \right\}
\end{split}
\end{equation}
for even $n$, and something quite similar for odd $n$.  

Now, we note that
\begin{equation}
\begin{split}
	 \Delta t_{\ell-1} \sum_{k=0}^{n/2 - 1} (X^{\ell-1}_k - \mathcal{X}^{\ell-1}_k) &=  \frac{\Delta t_{\ell-1}}{2} \sum_{k=0}^{n - 1} (X^{\ell-1}_k - \mathcal{X}^{\ell-1}_k) + \frac{\Delta t_{\ell-1}}{4} \left\{ (X^{\ell-1}_0 - \mathcal{X}^{\ell-1}_0) - (X^{\ell-1}_{n/2} - \mathcal{X}^{\ell-1}_{n/2}) \right\} \\
	 &= \Delta t_{\ell} \sum_{k=0}^{n - 1} (X^{\ell-1}_k - \mathcal{X}^{\ell-1}_k) + \frac{\Delta t_{\ell}}{2} (X^{\ell-1}_{n/2} - \mathcal{X}^{\ell-1}_{n/2})
\end{split}
\end{equation}
because of the linear interpolation definitions.  In addition, we have
\begin{equation}
\begin{split}
	\Delta t_\ell \sum_{k=0}^{n-1} \left\{\widehat{X}^\ell_k - \E [\mathcal{X}^\ell_k] \right\} &= \Delta t_{\ell-1} \sum_{k=0}^{n/2 - 1} \left\{ \widehat{X}^{\ell-1}_k - \E[\mathcal{X}^{\ell-1}_k] \right\} \\
	&+ \Delta t_\ell \sum_{k=0}^{n-1} \frac{1}{N_\ell} \sum_{i=1}^{N_\ell} \left\{ (X^\ell_k - \mathcal{X}^\ell_k) - (\tilde{X}^{\ell}_{k/2} - \mathcal{X}^{\ell-1}_{k/2}) \right\} \\
	&+ \frac{\Delta t_\ell}{2} \left\{ (\widehat{X}^{\ell-1}_{n/2} - \E[\mathcal{X}^{\ell-1}_{n/2}]) - (\widehat{X}^{\ell-1}_0 - \E [\mathcal{X}^{\ell-1}_0]) \right\} \\
	&+ \Delta t_\ell \sum_{k=0}^{n-1} \left\{ \frac{1}{N_\ell} \sum_{i=1}^{N_\ell} (\mathcal{X}^\ell_k - \mathcal{X}^{\ell-1}_{k/2}) - \E \left[ \mathcal{X}^\ell_k - \mathcal{X}^{\ell-1}_{k/2} \right] \right\}
\end{split}
\end{equation}
Since we've already concluded that $\tilde{X}^\ell_{k/2}$ and $X^{\ell-1}_{k/2}$ are identically distributed, we can replace one by the other inside expectations.  Defining $\delta^\ell_n = (X^\ell_n - \mathcal{X}^\ell_n) - (X^{\ell-1}_{n/2} - \mathcal{X}^{\ell-1}_{n/2})$, we have
\begin{equation} \label{thingy}
\begin{split}
	\delta^\ell_n &= A \Delta t_\ell \sum_{k=0}^{n-1} \delta^\ell_k - \frac{\Delta t_\ell}{2}(X^{\ell-1}_{n/2} - \mathcal{X}^{\ell-1}_{n/2}) + \Delta t_\ell \sum_{k=0}^{n-1} \nu^\ell_k \\
	&+ B\Delta t_\ell \sum_{k=0}^{n-1} \frac{1}{N_\ell} \sum_{i=1}^{N_\ell} \delta^{\ell,i}_k + \frac{\Delta t_\ell}{2} \left\{ (\widehat{X}^{\ell-1}_{n/2} - \E[\mathcal{X}^{\ell-1}_{n/2}]) - (\widehat{X}^{\ell-1}_0 - \E [\mathcal{X}^{\ell-1}_0]) \right\},
\end{split}
\end{equation}
where the $\nu^\ell_k$, defined by
\begin{equation}
	\nu^\ell_k = \frac{1}{N_\ell} \sum_{i=1}^{N_\ell} (\mathcal{X}^\ell_k - \mathcal{X}^{\ell-1}_{k/2}) - \E \left[ \mathcal{X}^\ell_k - \mathcal{X}^{\ell-1}_{k/2} \right],
\end{equation}
are random variables with zero mean and variance $O(\Delta t_\ell)/N_\ell$ by standard arguments.  

We take the cartesian norm and mean of (\ref{thingy}) and, after taking advantage of the exchangeability of different samples, have
\begin{equation}
\begin{split}
	\E \norm{ \delta^\ell_n } &\leq (|A| + |B|) \Delta t_\ell \sum_{k=0}^{n-1} \E \norm{ \delta^\ell_k } + K \sqrt{\frac{\Delta t_\ell}{N_\ell}} \\
	&+ \frac{\Delta t_\ell}{2} \left\{ \E \norm{X^{\ell-1}_{n/2} - \mathcal{X}^{\ell-1}_{n/2}} + \E \norm{\widehat{X}^{\ell-1}_{n/2} - \E \left[ \mathcal{X}^{\ell-1}_{n/2}\right]} \right\}.
\end{split}
\end{equation}
Importantly, this bound holds - albeit with slightly different constants - for all $n$, though we have only derived it for even $n$.  Deriving its analogue for odd $n$ is directly analogous, so we omit that computation here.  

Applying a discrete Gr\"{o}nwall inequality (see e.g.\ \cite{pachpatte2001inequalities}) gives
\begin{equation} \label{deltabnd}
	\E \norm{ \delta^\ell_n } \leq (1 + C T e^{C T}) \left[ K\frac{\Delta t_\ell}{\sqrt{N_\ell}} + \frac{\Delta t_\ell}{2} \left\{ \E \norm{X^{\ell-1}_{n/2} - \mathcal{X}^{\ell-1}_{n/2}} + \E \norm{\widehat{X}^{\ell-1}_{n/2} - \E \left[ \mathcal{X}^{\ell-1}_{n/2}\right]} \right\} \right],
\end{equation}
where $C \coloneqq |A| + |B|$.  

With this in hand, we note that a trivial telescoping sum implies that
\begin{equation}
	\max_n \E \norm{X^\ell_n - \mathcal{X}^\ell_n} \leq \max_n \E \norm{X^0_n - \mathcal{X}^0_n} + \sum_{m=1}^\ell \max_n \E \norm{\delta^\ell_n}.
\end{equation}
Defining $\epsilon^\ell = \max_n \E \norm{X^\ell_n - \mathcal{X}^\ell_n}$, $\widehat{\epsilon}^\ell = \max_n \E \norm{\widehat{X}^\ell_n - \E [\mathcal{X}^\ell_n]}$ and using (\ref{deltabnd}), we have
\begin{equation}
	\epsilon^\ell \leq K \left( \frac{1}{\sqrt{N_0}} + \sum_{m=1}^\ell \sqrt{\frac{\Delta t_m}{N_m}} \right) + \frac{1}{2} \sum_{m=1}^{\ell-1} \Delta t_m \left\{ \epsilon^m + \widehat{\epsilon}^m \right\}.
\end{equation}
Furthermore, by the definition of $\widehat{X}$, we have
\begin{equation}
	\widehat{\epsilon}^m \leq \widehat{\epsilon}^{m-1} + \epsilon^m + \epsilon^{m-1} \leq \epsilon^m + 2 \epsilon^{m-1} + ... + 2\epsilon^0
\end{equation}
We thus have
\begin{equation} \label{newbnd}
\begin{split}
	\epsilon^\ell &\leq K \left( \frac{1}{\sqrt{N_0}} + \sum_{m=1}^\ell \frac{\Delta t_m}{\sqrt{N_m}} \right) + \sum_{m=1}^{\ell-1} \Delta t_m \left\{ \epsilon^m + \sum_{r=1}^{m-1} \epsilon^r \right\} \\
	&\leq K \left( \frac{1}{\sqrt{N_0}} + \sum_{m=1}^\ell \frac{\Delta t_m}{\sqrt{N_m}} \right) + \sum_{m=1}^{\ell-1} \Delta t_m \sum_{r=1}^{m} \epsilon^r.
\end{split}
\end{equation}
since $\epsilon^0 \leq K_0/\sqrt{N_0}$.  

This clearly implies
\begin{equation}
	\max_{r\leq \ell} \epsilon^r \leq K \left( \frac{1}{\sqrt{N_0}} + \sum_{m=1}^\ell \frac{\Delta t_m}{\sqrt{N_m}} \right) + \sum_{m=1}^{\ell-1} m\Delta t_m \max_{r\leq m} \epsilon^r.
\end{equation}
Applying discrete Gr\"{o}nwall to this over $\ell$ gives
\begin{equation}
	\max_{r\leq \ell} \epsilon^r \leq K \left( \frac{1}{\sqrt{N_0}} + \sum_{m=1}^\ell \sqrt{\frac{\Delta t_m}{N_m}} \right) \left( 1 + \sum_{m=1}^{\ell-1} m\Delta t_m \exp\left\{ \sum_{j=1}^{m-1} j \Delta t_j \right\} \right)
\end{equation}
Bounding sums by their infinite analogues, which are convergent, we finally have
\begin{equation}
	\max_{r\leq \ell} \max_n \E \norm{X^r_n - \mathcal{X}^r_n} \leq K' \left( \frac{1}{\sqrt{N_0}} + \sum_{m=1}^\ell \sqrt{\frac{\Delta t_m}{N_m}} \right).
\end{equation}

\section{Proof of Theorem 3}

We will do the proof in the autonomous case - i.e.\ neither $\alpha$ nor $\beta$ has explicit dependence on $t$.  The generalization to nonautonomy is straightforward, but the presentation is murkier due to the additional terms.  Throughout the proof, $K$ will denote a generic constant which may depend on $T$ and $p$, as well as the Lipshitz constant in assumption 1 and the bounds on the norms in assumption 2.  It will not, however, depend on $\ell$ or $n$.  

We note first that
\begin{equation} \label{2step}
\begin{split}
	X_{n+2}^{\ell} &= X_n^{\ell} + \left( \alpha_n^{\ell} + \alpha_{n+1}^{\ell} \right)\Delta t_{\ell} + \beta_n^{\ell} \Delta W^{\ell}_n + \beta_{n+1}^{\ell} \Delta W^{\ell}_{n+1} \\
	&=  X_n^{\ell} + \alpha_n^{\ell} \Delta t_{\ell-1} + \beta_n^{\ell} \Delta \tilde{W}_{n/2}^\ell + \left( \alpha_{n+1}^{\ell}  - \alpha_{n}^{\ell} \right) \Delta t_{\ell} + (\beta_{n+1}^{\ell} - \beta_n^{\ell})\Delta W_{n+1}^{\ell}
\end{split}
\end{equation}
where we've suppressed the common $i$ superscript.  Subtracting the analogous expression for $\tilde{X}_{(n+2)/2}^\ell$, we have 
\begin{equation}
\begin{split}
	X^{\ell}_{n+2} - \tilde{X}^\ell_{(n+2)/2} = X^{\ell}_{n} - \tilde{X}^\ell_{n/2} &+ (\alpha^{\ell}_n - \tilde{\alpha}^\ell_{n/2})\Delta t_{\ell-1} + (\beta^{\ell}_n - \tilde{\beta}^\ell_{n/2}) \Delta \tilde{W}^\ell_{n/2} \\
	&+ (\alpha^{\ell}_{n+1} - \alpha^{\ell}_n) \Delta t_{\ell} + (\beta^{\ell}_{n+1} - \beta^{\ell}_n)\Delta W_{n+1}^{\ell}.  
\end{split}
\end{equation}
Iterating, taking norms followed by maximums, followed by expectations, and defining $\delta_n^\ell(p) = \E \left[ \max_{m\leq n}\norm{X^{\ell}_m - \tilde{X}^\ell_{m/2}}^p \right]$, we have
\begin{equation}  \label{normsum}
\begin{split}
	\delta^\ell_{n+2}(p) &\leq \Delta t_{\ell-1}^p \E \left[ \max_{m\leq n} \norm{ \sum_{k \textrm{ even}}^m (\alpha^{\ell}_k - \tilde{\alpha}^\ell_{k/2})}^p \right] \\
	&+ \E \left[ \max_{m\leq n} \norm{ \sum_{k \textrm{ even}}^m (\beta^{\ell}_k - \tilde{\beta}^\ell_{k/2}) \Delta \tilde{W}^\ell_{k/2} }^p \right] \\
	&+ \Delta t_{\ell}^p \E \left[ \max_{m\leq n} \norm{ \sum_{k \textrm{ even}}^m (\alpha^{\ell}_{k+1} - \alpha^{\ell}_k) }^p \right] \\ 
	&+ \E \left[ \max_{m\leq n} \norm{ \sum_{k \textrm{ even}}^m (\beta^{\ell}_{k+1} - \beta^{\ell}_k)\Delta W_{k+1}^{\ell} }^p \right].  
\end{split}
\end{equation}
We label the four expectations above I-IV and analyze them separately.  Terms I and III have in common the absence of Brownian motion, so we investigate them first.  

\subsection{Bounding I} 
We have
\begin{equation}
	\textrm{I} \leq (n/2)^{p-1} \sum_{k \textrm{ even}}^n \E \left[ \norm{ \alpha_k^{\ell} - \tilde{\alpha}^\ell_{k/2} }^p \right]
\end{equation}
To bound the expectation inside the sum, we write
\begin{equation} \label{term1}
\begin{split}
	\E \left[ \norm{ \alpha_k^{\ell} - \tilde{\alpha}_{k/2}^{\ell}}^p \right] &\leq K \left\{ \E \left[ \norm{X_k^{\ell} - \tilde{X}_{k/2}^\ell}^p \right] + \E \left[ \norm{\widehat{R}_k^{\ell} - \widehat{R}_{k/2}^{\ell-1}}^p \right]  \right\},
\end{split}
\end{equation}
by virtue of $\alpha$ being Lipschitz in both arguments.  By definition, 
\begin{equation} \label{Rbound1}
\begin{split}
	\E \left[ \norm{\widehat{R}_k^{\ell} - \widehat{R}_{k/2}^{\ell-1}}^p \right] &= 
	\frac{1}{N_{\ell}^p} \E \left[ \norm{\sum_{j=1}^{N_{\ell+1}} \left\{R\left(X_k^{\ell,j} \right) - R\left( \tilde{X}_{k/2}^{\ell,j} \right) \right\}}^p \right] \\
	&\leq \frac{K}{N_{\ell}} \sum_{j=1}^{N_{\ell}} \E \left[ \norm{X_k^{\ell,j} - \tilde{X}_{k/2}^{\ell,j}}^p \right] \\
	&\leq K \E \left[ \norm{X_k^{\ell} - \tilde{X}_{k/2}^{\ell}}^p \right],
\end{split}
\end{equation}
where the last line comes from the exchangeability of the samples.  Substituting this into (\ref{term1}), we have
\begin{equation} \label{usfl1}
	\E \left[ \norm{ \alpha_k^{\ell+1} - \alpha_{k/2}^{\ell}}^p \right] \leq K \E \left[ \norm{X_k^{\ell+1} - X_{k/2}^\ell}^p \right].
\end{equation}
Notice that this bound still holds when $\alpha$ is replaced by $\beta$, since they satisfy all the same assumptions - this will be used in bounding II.  Returning to I, we now have
\begin{equation} \label{term1simp}
	\textrm{I} \leq K (n/2)^{p-1} \sum_{k \textrm{ even}}^n \E \left[ \norm{X_k^{\ell} - \tilde{X}_{k/2}^\ell}^p \right] \leq K n^{p-1} \sum_{k \textrm{ even}}^n \delta^\ell_k(p)
\end{equation}

\subsection{Bounding III}
Next, we turn our attention to III.  By again taking advantage of the fact that $\alpha$ is Lipschitz, we have
\begin{equation} \label{pt2}
\begin{split}
	\textrm{III} &\leq (n/2)^{p-1}\E \left[ \sum_{k \textrm{ even}}^n \norm{ \alpha_{k+1}^{\ell}  - \alpha_{k}^{\ell}}^p \right] \\
	&\leq K n^{p-1} \sum_{k \textrm{ even}}^n\left\{ \E \left[  \norm{ X_{k+1}^{\ell} - X_k^{\ell} }^p \right] + \E \left[ \norm{ \widehat{R}_{k+1}^{\ell} - \widehat{R}_k^{\ell} }^p \right] \right\}.
\end{split}
\end{equation}
Note next that
\begin{equation} \label{simpbnd}
\begin{split}
	\E \left[ \norm{ X_{k+1}^{\ell+1} - X_k^{\ell+1} }^p \right] &\leq 2^{p-1} \E \left[ \norm{\alpha_k^{\ell+1} }^p \right] \Delta t_{\ell+1}^p + 2^{p-1} \E \left[ \norm{\beta_k^{\ell+1} }^p \right] \Delta t_{\ell+1}^{p/2} \\
	&\leq K (\Delta t_\ell^p + \Delta t_\ell^{p/2}),
\end{split}
\end{equation}
where the second line follows from the bounded expectations of the coefficients (assumption 2).

To attack the last norm in (\ref{pt2}), we note that 
\begin{equation}
	\widehat{R}_{k+1}^{\ell} = \frac{1}{2} \left( \widehat{R}_{k/2}^{\ell-1} + \widehat{R}_{(k+2)/2}^{\ell-1} \right) + \frac{1}{N_{\ell}} \sum_{j=1}^{N_{\ell}} \left\{ R_{k+1}^{\ell,j} - \frac{1}{2} \left( \tilde{R}^{\ell,j}_{k/2} + \tilde{R}^{\ell,j}_{(n+2)/2} \right) \right\},
\end{equation}
so that we may write
\begin{equation} \label{complicatedbnd}
\begin{split}
	\E \left[ \norm{ \widehat{R}_{k+1}^{\ell} - \widehat{R}_k^{\ell} }^p \right] &\leq 
	\frac{1}{2} \E \left[ \norm{ \widehat{R}_{(k+2)/2}^{\ell-1} - \widehat{R}_{k/2}^{\ell-1} }^p \right] \\
	&+ 2^{p-1}\E \left[ \norm{ \left\{ R^{\ell}_{k+1} - R^{\ell}_k - \frac{1}{2} \left( \tilde{R}^{\ell}_{k/2} + \tilde{R}^{\ell}_{(k+2)/2} \right)  + \tilde{R}^{\ell}_{k/2} \right\}}^p \right] \\
	&\leq \frac{1}{2} \E \left[ \norm{ \widehat{R}_{(n+2)/2}^{\ell} - \widehat{R}_{n/2}^{\ell} }^p \right] \\
	&+ K \left\{ \E \left[ \norm{ X^{\ell}_{k+1} - X^{\ell}_k }^p \right] +  \E \left[ \norm{ \tilde{X}^\ell_{(n+2)/2} - \tilde{X}^\ell_{n/2} }^p \right] \right\} \\
	&\leq  \frac{1}{2} \E \left[ \norm{ \widehat{R}_{(n+2)/2}^{\ell} - \widehat{R}_{n/2}^{\ell} }^p \right] + K \left\{ \Delta t_\ell^p + \Delta t_\ell^{p/2} \right\}
\end{split}
\end{equation}
where the last inequality results from the Lipschitz bound on $R$ and (\ref{simpbnd}).

In the end, in (\ref{complicatedbnd}), we have a bound on a quantity in terms of its analogue one level lower.  We can thus iterate the process, finding
\begin{equation} \label{reducedbnd}
	\E \left[ \norm{ \widehat{R}_{k+1}^{\ell} - \widehat{R}_k^{\ell} }^p \right] \leq
	2^{-(\ell + 1)} \E \left[ \norm{ \widehat{R}_1^0 - \widehat{R}_0^0}^p \right] + K \sum_{m=0}^\ell 2^{m - \ell} (\Delta t_m^p + \Delta t_m^{p/2}).
\end{equation}
We now specify $p=2$ and note that for $\Delta t_\ell < 1$, 
\begin{equation}
	\E \left[ \norm{ \widehat{R}_{k+1}^{\ell} - \widehat{R}_k^{\ell} }^2 \right] \leq K \Delta t_\ell (\ell + 1).
\end{equation}
For $p=1$, we have
\begin{equation}
	\E \left[ \norm{ \widehat{R}_{n+1}^{\ell+1} - \widehat{R}_n^{\ell+1} } \right] \leq K \Delta t_\ell^{1/2}.
\end{equation}
And finally, for $2<p \leq p_{\textrm{max}}$, 
\begin{equation}
	\E \left[ \norm{ \widehat{R}_{n+1}^{\ell+1} - \widehat{R}_n^{\ell+1} }^p \right] \leq K \Delta t_\ell.
\end{equation}
As such, we define 
\begin{equation}
	\epsilon_\ell(p) \coloneqq \left\{
	\begin{array}{lcr}
		\Delta t_\ell^{1/2} & : & p=1 \\
		\Delta t_\ell \left| \log \Delta t_\ell \right| & : & p=2 \\
		\Delta t_\ell & : & p>2
	\end{array} \right.
\end{equation}
so that upon plugging into (\ref{pt2}), we have
\begin{equation} \label{PT2}
	\E \left[ \norm{ \alpha_{n+1}^{\ell+1}  - \alpha_{n}^{\ell+1}}^p \right] \leq K \epsilon_\ell(p).
\end{equation}
Notice that this bound also applies when $\alpha$ is replaced by $\beta$, since they satisfy all the same assumptions - this will be used in bounding II and IV.  For III, we have 
\begin{equation} \label{IIIbnd}
	\textrm{III} \leq K n^p \epsilon_\ell(p)
\end{equation}

\subsection{Bounding II and IV}
Bounding II and IV is qualitatively different from I and III because the arguments of the norms are martingales.  Recall that the discrete Burkholder-Davis-Gundy inequality requires that the maximum of a martingale is bounded by its quadratic variation.  More precisely,
\begin{equation}
	\E \left[ \max_{m\leq n} \norm{M(m)}^p \right] \leq K \E \left[ \norm{ \sum_{k=1}^n (M(k) - M(k-1))^2 }^{p/2} \right]
\end{equation}
for any discrete martingale $M$.  It follows that
\begin{equation} \label{IIbnd}
\begin{split}
	\textrm{II} &= \E \left[ \max_{m\leq n} \norm{ \sum_{k \textrm{ even}}^m (\beta^{\ell+1}_k - \beta^\ell_{k/2}) \Delta W^\ell_{k/2} }^p \right] \\ 
	&\leq K \E \left[ \norm{ \sum_{k \textrm{ even}}^n \left( (\beta^{\ell+1}_k - \beta^\ell_{k/2}) \Delta W^\ell_{k/2} \right)^2 }^{p/2} \right] \\
	&\leq K n^{p/2-1} \sum_{k \textrm{ even}}^n \E \left[ \norm{ X^{\ell+1}_{k} - X^{\ell}_{k/2} }^p \right] \E \left[ \norm{ \Delta W^{\ell+1}_k }^p \right] \\
	&\leq K(n\Delta t_\ell)^{p/2-1} \Delta t_\ell \sum_{k \textrm{ even}}^n \delta_k^\ell(p),
\end{split}
\end{equation}
where we've used (\ref{usfl1}) with $\alpha \rightarrow \beta$.  

Similarly, for IV we have
\begin{equation} \label{IVbnd}
\begin{split}
	\textrm{IV} &= \E \left[ \norm{ \sum_{k \textrm{ even}}^n (\beta^{\ell+1}_{k+1} - \beta^{\ell+1}_k)\Delta W_{k+1}^{\ell+1} }^p \right] \\
	&\leq K n^{p/2-1} \Delta t^{p/2} \sum_{k \textrm{ even}}^n \E \left[ \norm{ \beta^{\ell+1}_{k+1} - \beta^{\ell+1}_{k} }^p \right] \\
	&\leq K (n\Delta t_\ell)^{p/2} \epsilon_\ell(p) 
\end{split}
\end{equation}
where we've used (\ref{PT2}) with $\alpha \rightarrow \beta$ to get the last line.  

\subsection{Discrete Gr\"onwall}
Finally, substituting (\ref{term1simp}), (\ref{IIIbnd}), (\ref{IIbnd}) and (\ref{IVbnd}) into (\ref{normsum}), we have
\begin{equation} \label{gronwallleft}
	\delta^\ell_{n+2}(p) \leq K \left\{ (n\Delta t_\ell)^{p-1} \Delta t_\ell \sum_{k \textrm{ even}}^n \delta^\ell_k(p) + (n \Delta t_\ell)^p\epsilon_\ell(p) \right\}
\end{equation}
Since, $n\Delta t_\ell$ is bounded by the simulation time $T$, applying the appropriate discrete Gr\"onwall inequality and noting that $\ell = \log_2 (T/\Delta t_\ell)$ gives the desired result.  Namely, 
\begin{equation}
	\delta^\ell_{n}(p)= O\left( \epsilon_\ell(p) \right) \qquad \textrm{for all } n \leq T/\Delta t_{\ell+1}.
\end{equation}
This completes the proof. \qed

\end{appendix}

\bibliographystyle{plain}
\bibliography{GenMcKeanVlasovArxiv}

\end{document}